\documentclass[11pt]{amsart}
\usepackage{geometry}               
\usepackage{graphicx}
\usepackage{amssymb}
\usepackage{epstopdf}
\newtheorem{theorem}{Theorem}
\newtheorem{corollary}[theorem]{Corollary}
\newtheorem{lemma}[theorem]{Lemma}
{\theoremstyle{remark}
\newtheorem{remark}{Remark}
\newtheorem{definition}{Definition}
\newtheorem{example}{Example}[section]}
\newcommand{\RR}{\mathbb{R}}
\newcommand{\NN}{\mathbb{N}}
\newcommand{\ZZ}{\mathbb{Z}}
\DeclareMathOperator*{\osc}{osc}
\newcommand{\tx}[1]{{\text{\rm #1}}}

\title[Lipschitz bounds for quasilinear parabolic equations]{Lipschitz bounds for solutions of quasilinear parabolic equations in one space variable}
\author{Ben Andrews}
\thanks{Research partially supported by a Discovery Grant of the Australian Research Council}

\address{Centre for Mathematics and its Applications,  Australian National University, A.C.T. 0200, Australia}
\email{Ben.Andrews@maths.anu.edu.au}
\author{Julie Clutterbuck}
\address{Centre for Mathematics and its Applications,  Australian National University, A.C.T. 0200, Australia}
\email{Julie.Clutterbuck@maths.anu.edu.au}
 
\begin{document}

\begin{abstract}
We bound the modulus of continuity of solutions to quasilinear parabolic equations in one space variable in terms of the initial modulus of continuity and elapsed time.  In particular we characterize those equations for which the Lipschitz constants of solutions can be bounded in terms of their initial oscillation and elapsed time.
\end{abstract}

\maketitle

\section{Introduction}\label{sec:intro}

In this paper we investigate the extent to which degenerate nonlinear parabolic
equations smooth out irregularities in the initial data.  This is a well-known phenomenon for the
classical heat equation:  Initial data which are very singular (such as in Sobolev spaces of arbitrary
negative exponent) give rise to solutions which are $C^{\infty}$ for any positive time.   This ceases to be true for more nonlinear equations, particularly in cases where the equation becomes degenerate when the gradient becomes large.  Our aim is to delineate clearly when such flows give rise to classical solutions from initial data which are merely continuous, and conversely to characterise the modulus of continuity required on the initial data to guarantee a classical solution for positive times.  

In this first paper we give a thorough treatment of equations with one spatial variable, where the situation can be completely understood.  A subsequent paper will extend these methods to higher dimensions, and show that for many equations of interest the behaviour is determined by a suitable one-dimensional problem.  Thus the results of this paper, while of some interest in their own right, also serve as a foundation for our work on equations with several spatial variables.

The argument employed in this paper is based on a method used by Kruzhkov
\cite{kruzhkov:nonlinear}:
If $u$ is a solution of a parabolic equation in one variable, so that
$$
u_t = au_{xx} + bu_x +cu + f
$$ 
where $a$, $b$ and $c$ and $f$ are bounded, $a$ is strictly positive, and $u$ is 
bounded (say $|u|\leq M$), then $w(x,y,t)=u(y,t)-u(x,t)$
satisfies a parabolic equation in two spatial variables:
\begin{align*}
w_t(x,y,t) &= a(y,t)u_{yy}(y,t)+b(y,t)u_y(y,t)+c(y,t)u(y,t)+f(y,t)\cr
&\phantom{==}-a(x,t)u_{xx}(x,t)-b(x,t)u_x(x,t)-c(x,t)u(x,t)-f(x,t)\cr
&=a(y,t)w_{yy}+a(x,t)w_{xx}+b(y,t)w_y+b(x,t)w_x+F(x,y,t)\cr
\end{align*}
where $|F(x,y,t)|=|c(y,t)u(y,t)-c(x,t)u(x,t)+f(y,t)-f(x,t)|\leq 2M\sup|c|+2\sup|f|$.  

The important
simplification that is achieved by this is the following:  If $u$ is
defined on an interval of the real line, then $w$ can be
defined on an open set in the half-plane $\{y\geq x\}$ in ${\mathbb R}^2$,
and we have
$w=0$ on the boundary
$\{y=x\}$. A boundary gradient estimate for $w$ along
this line implies a global gradient estimate for $u$.  But
boundary gradient estimates can be proved using the parabolic maximum
principle simply by constructing suitable supersolutions near
any boundary point.   In the present example a barrier can be constructed 
at a boundary point $(z,z)$ by taking the form
$$
\psi(x,y,t) = \min\left\{A(x+y-z)^{2}+
Be^{-C(y-x)}\text{\rm erf}\left\{\frac{D(y-x)}{\sqrt{t}}\right\},M\right\}
$$
and choosing $A$ large (compared to $M/d^{2}$, where $d$ is the distance from $z$ to the 
boundary of the interval), $B$ large compared to $M$, $D$ large (compared to $1/a$), 
$C$ large (compared to $|b|/a$), $D$ large (compared to $1/a$), 
and the time interval small.

This yields interior gradient estimates for solutions of any such
equation, of the form
$$
|u'(x,t)|^2\leq CM^2\left(t^{-1}+d^{-2}\right)
$$
where $d$ is the distance from $x$ to the boundary of the domain.

In this paper we will sharpen the above argument to give optimal estimates, particularly in the case where the coefficients depend only on the gradient.  More generally, we will provide the best possible control on the modulus of continuity of solutions for positive times, in terms of the initial modulus of continuity and the elapsed time:   In Section \ref{sec:estimate} we prove the basic result, which bounds the modulus of continuity of any solution in terms of any supersolution of the same equation with the initial modulus of continuity as initial data.   The best such estimate is then given by the infimum over all supersolutions, which may be considered a viscosity solution of the equation.  In Section \ref{sec:optimal} we prove that this estimate is sharp:   Among all solutions of the equation with initial data satisfying a given bound on the modulus of continuity,  the supremum of the moduli of continuity is precisely the bound obtained in the previous section.  In Section \ref{sec:suff} we give concrete estimates on the modulus of continuity by considering supersolutions constructed from translating solutions of the flow (the strucure of which we first develop in Section
\ref{sec:trans}).  The result of this is a necessary and sufficient condition on the coefficients of the equation for the existence of a Lipschitz bound for positive times in terms of oscillation bounds at the initial time.  More generally,  we provide a sufficient condition in terms of the coefficients of the equation for a given modulus of continuity to imply Lipschitz bounds for positive times.  We place in an appendix some results for uniformly parabolic equations with coefficients depending only on the gradient.

\section{Estimate on the modulus of continuity}\label{sec:estimate}

In this section we prove the basic estimate, which controls the spatial modulus of continuity of any solution in terms of any supersolution with initial data determined by the initial modulus of continuity.
We illustrate this first for the graphical curve-shortening flow, and then treat more general one-dimensional equations.  While a direct generalisation of Kruzhkov's argument does provide a bound, our sharp estimate requires the use of the full Hessian matrix of the function of two spatial variables.

Let $u: {\mathbb R}\times[0,T)\to{\mathbb R}$ be a smooth solution of the
curve-shortening flow
\begin{equation}\label{eq:csf}
\frac{\partial u}{\partial t}= 
\frac{u''}{1+{u'}^2}.
\end{equation}
Suppose $\osc u=\sup u-\inf u\leq M$ and $u(x+L,t)=u(x,t)$ for every 
$x\in{\RR}$ and $t\geq 0$.  

\begin{theorem} \label{Theorem 1}
For all $x\neq y$ in ${\RR}$ and $t>0$,
$$
|u(y,t)-u(x,t)|\leq 2M\varphi\left(\frac{|y-x|}{2M},\frac{t}{M^2}\right)
$$
where $\varphi: [0,\infty)\times (0,\infty)\to{\RR}$
is the solution of \eqref{eq:csf}  satisfying
$\varphi(x,t)\to \frac12$ as $t\to 0$ for $x>0$, $\varphi(0,t)=0$ for $t>0$, 
and $\varphi(x,t)\to \frac12$ as $x\to\infty$ for any $t>0$.
In particular, 
$$
|Du(x,t)|\leq C_1\left(1+ t^{3/2}\exp(C_2/t)\right)
$$
for some constants $C_1$ and $C_2$ depending only on $M$.
\end{theorem}

The existence of $\varphi$ follows from the results of Ecker and Huisken \cite{eh:interior}, since
the initial condition can be rotated to be a Lipschitz graph.  The strong maximum principle (applied to the rotated graph) implies that  $\varphi$ has bounded gradient for 
positive times.

\begin{proof}
By replacing $u$ by $\frac{1}{M}u(Mx,M^2t)$ we can assume $M=1$.  Let $\varepsilon>0$, and define
$$
Z(x,y,t)=u(y,t)-u(x,t)-2\varphi\left(\frac{|y-x|}{2},t\right)-\varepsilon(1+t)
$$ 
on $S=\{(x,y,t):\ x\leq y\leq x+L,\ t > 0\}$.  $Z$ is negative near the diagonal $\{y=x\}$ and near $\{y=x+L\}$ since $u$ is continuous and periodic, and near $t=0$ since $\varphi$ approaches $1/2$ locally uniformly away from the diagonal, while $|u(y,t)-u(x,t)|\leq 1$.

Suppose $Z$ is not negative on $S$.  Then there exists a first time $t_0>0$ and a point $(x_0,y_0)$ with $x_0<y_0<x_0+L$ such that $Z(x_0,y_0,t_0)=0$.
At this point we have 
\begin{equation}\label{eq:xderivcond}
0=\frac{\partial Z}{\partial x} = -u'(x_0,t_0)+\varphi'\left(\frac{y_0-x_0}{2},t_0\right)
\end{equation}
where $\varphi'$ denotes the derivative of $\phi$ in the first argument,
and
\begin{equation}\label{eq:yderivcond}
0=\frac{\partial Z}{\partial y} = u'(y_0,t_0)-\varphi'\left(\frac{y_0-x_0}{2},t_0\right),
\end{equation}
and the Hessian matrix is negative semi-definite
\begin{equation}\label{eq:Hessiancond}
0\geq\bmatrix \frac{\partial^2Z}{\partial x^2}&\frac{\partial^2Z}{
\partial x\partial y}\cr
\frac{\partial^2Z}{\partial y\partial x}&\frac{\partial^2Z}{\partial y^2}\cr
\endbmatrix
\!=\!\bmatrix
-u''(x,t)-\frac12\varphi''\left(\frac{y-x}{2},t\right) & \frac12\varphi''\left(\frac{y-x}{2},t\right)\cr
 \frac12\varphi''\left(\frac{y-x}{2},t\right) & u''(y,t)-\frac12\varphi''\left(\frac{y-x}{2},t\right)\cr
\endbmatrix.
\end{equation}
Since $u$ satisfies Equation \eqref{eq:csf},
{\allowdisplaybreaks we have at the maximum of $Z$
\begin{align*}
0\leq\frac{\partial Z}{\partial t}(x,y,t)
&=\frac{\partial u}{\partial t}(y,t)-
\frac{\partial u}{\partial t}(x,t)-2\frac{\partial\varphi}{\partial t}\left(\frac{y-x}{2},t\right)-\varepsilon\cr
&=\frac{u''(y,t)}{ 1+(u'(y,t))^2}-\frac{u''(x,t)}{ 1+(u'(x,t))^2}
-2\frac{\partial\varphi}{\partial t}\left(\frac{y-x}{2},t\right)-\varepsilon\cr
&=\frac{\frac{\partial^2Z}{\partial y^2}+\frac12\varphi''}{ 1+(\varphi')^2}
+\frac{\frac{\partial^2Z}{\partial x^2}+\frac12\varphi''}{ 1+(\varphi')^2}+2c\left(\!\!\frac{\partial^2Z}{\partial x\partial
y}-\frac12\varphi''\right)
-2\frac{\partial\varphi}{\partial t}-\varepsilon\cr 
&=\text{\rm Tr}
\left(
\bmatrix \frac{1}{ 1+(\varphi')^2}&c\cr c&\frac{1}{ 1+(\varphi')^2}\cr\endbmatrix
\bmatrix \frac{\partial^2Z}{\partial x^2}&\frac{\partial^2Z}{
\partial x\partial y}\cr
\frac{\partial^2Z}{\partial y\partial x}&\frac{\partial^2Z}{\partial y^2}\cr
\endbmatrix
\right)\cr
&\quad\null+\frac12\left(
\frac{2}{ 1+(\varphi')^2}-2c\right)\varphi''-2\frac{\partial\varphi}{\partial
t}-\varepsilon\cr
&\leq \frac12\left(
\frac{2}{ 1+(\varphi')^2}-2c\right)\varphi''-2\frac{\partial\varphi}{\partial
t}-\varepsilon
\cr
\end{align*}
provided the matrix }
$$
\bmatrix \frac{1}{ 1+(\varphi')^2}&c\cr c&\frac{1}{ 1+(\varphi')^2}\cr\endbmatrix
$$
is positive semi-definite.  Choosing $c=-1/(1+(\varphi')^2)$ we arrive at a contradiction, since $\varphi$ has been chosen to satisfy Equation \eqref{eq:csf}.  This contradicts the assumption that $Z$ is not negative on $S$.
Therefore $u(y,t)-u(x,t)< 2\varphi\left(\frac{|y-x|}{2},t\right)+\varepsilon(1+t)$ for every $\varepsilon>0$, and hence $u(y,t)-u(x,t)\leq 2\varphi\left(\frac{|y-x|}{2},t\right)$ for all $x\leq y\leq x+L$ and all $t>0$.  A similar argument proves the result for $y<x<y+L$.

The explicit gradient estimate in the theorem follows by comparing $\varphi$ with the function $\psi$ defined implicitly by
$$
\xi = t^{-1/2}
\left(\exp\left\{-\frac{|\psi(\xi,t)-1|^2}{
8t}\right\}-\exp\left\{-\frac{|\psi(\xi,t)+1|^2}{ 8t}\right\}\right).
$$
This is a supersolution for $t$ sufficiently small where $|\psi|\leq\frac12$, so
$0<\varphi(\xi,t)<\min\left\{\psi(\xi,t),\frac12\right\}$ for $\xi,t>0$.
In particular this implies that 
$$
\varphi'(0,t)\leq{2t^{3/2}}\exp\left\{\frac{1}{ 8t}\right\}
$$
for $t$ small enough, and this gives the bound for $|Du|$ in the Theorem.
\end{proof}

The method of proof above applies to quite general equations of the form
\begin{equation}\label{eq:1Dflow}
\frac{\partial u}{\partial t} = \alpha(u')u''
\end{equation}
where the coefficient $\alpha$ is continuous and positive.  We call a function \emph{regular} if it has continuous spatial derivatives up to second order and continuous first time derivative (see Appendix A for regularity statements for equations of this kind).  

\begin{definition}\label{def:modcont}
Let $u\in C(\RR)$ be periodic of period $L$, and let $\psi\in C(0,L/2)$ be positive. $\psi$ is a \emph{modulus of continuity} for $u$ if for all $0<y-x<L$,
\begin{equation}\label{eq:modofcont}
-2\psi\left(\frac{L+x-y}{2}\right)\leq u(y)-u(x)\leq 2\psi\left(\frac{y-x}{2}\right).
\end{equation}
\end{definition}
\begin{definition}\label{def:supersol}
Let $\psi\in C(0,L/2)$ be positive.  Let ${\mathcal S}_\psi$ be the space of functions  $\varphi\geq 0$ which are continuous on $[0,L/2]\times[0,T]\setminus\{(0,0),(L/2,0)\}$, regular on $(0,L/2)\times(0,T]$ with
$\frac{\partial\varphi}{\partial t}\geq \alpha(\varphi')\varphi''$, and have $\varphi(z,0)\geq \psi(z)$ for all $z\in(0,L/2)$.
\end{definition}

\begin{theorem}\label{thm:1Dgeneral}
Let $u$ be a regular $L$-periodic solution of  \eqref{eq:1Dflow} on $\RR\times[0,T]$, 
 and let $\psi$ be a modulus of continuity for $u(.,0)$.  Let $\varphi\in {\mathcal S}_\psi$. Then $\varphi(.,t)$ is a modulus of continuity for $u(.,t)$ for each $t>0$.
\end{theorem}

{\par\smallskip{\noindent\textbf{Remarks:}}
\begin{enumerate}
\item
The case $\psi(z) \equiv M$ bounds the modulus of continuity of solutions in terms of initial oscillation and elapsed time (as in Theorem \ref{Theorem 1}).
\item
If $\varphi(0,t)=0$ for $t>0$ then \eqref{eq:modofcont} implies $u'(z,t)\leq\varphi'(0,t)$ for every $z$, while if $\varphi(L/2,t)=0$ for $t>0$ then $u'(z,t)\geq \varphi'(L/2,t)$ for every $z$.
\item
If  $\varphi(0,t)=\varphi(L/2,t)=0$ for $t>0$ and $\varphi$ satisfies \eqref{eq:1Dflow} with $\varphi(z,0)=\psi(z)$, then the result is sharp, with equality in the case $u=\varphi$ (extended to be periodic). 
\end{enumerate}

\begin{proof}
Let $\varepsilon>0$, and
define $Z(x,y,t)=u(y,t)-u(x,t)-2\varphi\left(\frac{y-x}{2},t\right)-\varepsilon(1+t)$ 
on $\{(x,y,t):\ x<y<x+L,\ t > 0\}$.  $Z$ is negative for small $t>0$ (since $u$ and $\varphi$ are continuous), and negative where $y-x$ or $L+x-y$ are small for any $t$.
 If $Z$ is not everywhere negative then there exists $t_0>0$ and $x_0<y_0<x_0+L$ such that $Z(x_0,y_0,t_0)=\max\{Z(x,y,t):\ x<y<x+L,\ 0\leq t\leq t_0\}$.  At this point the first order conditions \eqref{eq:xderivcond} and \eqref{eq:yderivcond}  and
the Hessian condition \eqref{eq:Hessiancond} hold.
Since $u$ satisfies \eqref{eq:1Dflow},
{\allowdisplaybreaks 
\begin{align*}
\frac{\partial Z}{\partial t}(x,y,t)
&=\frac{\partial u}{\partial t}(y,t)-
\frac{\partial u}{\partial t}(x,t)-2\frac{\partial\varphi}{\partial t}\left(\frac{y-x}{2},t\right)-\varepsilon\cr
&=\alpha(u'(y,t))u''(y,t)-\alpha(u'(x,t))u''(x,t)
-2\frac{\partial\varphi}{\partial t}\left(\frac{y-x}{2},t\right)-\varepsilon\cr
&<\alpha(\varphi')\left(\frac{\partial^{2}Z}{\partial y^{2}}+\frac{\varphi''}{2}\right)
+\alpha(\varphi')\left(\frac{\partial^{2}Z}{\partial x^{2}}+\frac{\varphi''}{2}\right)+2c\left(\!\!\frac{\partial^2Z}{\partial x\partial
y}-\frac{\varphi''}{2}\right)-2\frac{\partial\varphi}{\partial t}\cr &=\text{\rm Tr}
\left(
\bmatrix \alpha(\varphi')&c\cr c&\alpha(\varphi')\cr\endbmatrix
\bmatrix \frac{\partial^2Z}{\partial x^2}&\frac{\partial^2Z}{
\partial x\partial y}\cr
\frac{\partial^2Z}{\partial y\partial x}&\frac{\partial^2Z}{\partial y^2}\cr
\endbmatrix
\right)+\left(\alpha(\varphi')-c\right)\varphi''-2\frac{\partial\varphi}{\partial
t}.
\cr
\end{align*}
The terms involving second derivatives of $Z$ are non-positive provided }
$\bmatrix \alpha(\varphi')&c\cr c&\alpha(\varphi')\cr\endbmatrix$
is positive semi-definite.  Choosing $c=-\alpha(\varphi')$ we arrive at a contradiction:
$$
0\leq \frac{\partial Z}{\partial t}< 2\left(\alpha(\varphi')\varphi''
-\frac{\partial\varphi}{\partial t}\right)\leq 0.
$$
It follows that $Z$ remains negative, and sending $\varepsilon$ to zero gives
the right-hand inequality in \eqref{eq:modofcont}.   Replacing $y$ by $x+L$ and $x$ by $y$ (and using the periodicity of $u$) yields the left-hand inequality.
\end{proof}

\section{Sharpness of the estimate}\label{sec:optimal}

Theorem \ref{thm:1Dgeneral} gave an estimate on the modulus of continuity of a solution of \eqref{eq:1Dflow} in terms of any supersolution $\varphi\in {\mathcal S}_\psi$, where $\psi$ is the initial modulus of continuity.

\begin{definition}\label{def:minsuper}
If $\psi\in C(0,L/2)$ is positive, then the \emph{minimal supersolution} of \eqref{eq:1Dflow} is
$$
\psi_+(z,t) = \inf\left\{\varphi(z,t):\ \varphi\in{\mathcal S}_\psi\right\}.
$$
\end{definition}

It is an immediate consequence of Theorem \ref{thm:1Dgeneral} that if $\psi$ is a modulus of continuity for the initial data of an $L$-periodic solution $u$ of \eqref{eq:1Dflow}, then $\psi_+(.,t)$ is a modulus of continuity for $u$ at any positive time $t$.  The aim of this section is to prove that the resulting estimate is sharp if $\psi$ is concave.

\begin{theorem}\label{thm:sharpbarrier}
Let $\psi$ be concave and positive on $(0,L/2)$.  Then $\psi_+(.,t)$ is concave for each $t\geq 0$, and $\psi_+$ is a regular solution of \eqref{eq:1Dflow} on $(0,L/2)\times(0,\infty)$.  Furthermore, there exists a sequence $v_k$ of regular solutions to \eqref{eq:1Dflow}, with
$v_k(.,t)$ concave and $v_k(0,t)=v_k(L/2,t)=0$ for each $t>0$, and $v_k(.,0)\leq\psi$, such that $v_k$ converges to $\psi_+$ locally uniformly in $(0,L/2)\times[0,\infty))$.
\end{theorem}

\begin{proof}
Choose a sequence of smooth concave functions $\psi_k$ on $[0,L/2]$ with $\psi_k(0)=\psi_k(L/2)=0$, such that 
\begin{equation}\label{eq:initialcomparison}
\psi_k(z)\leq \psi(z)\leq \frac{1+k}{k}\psi_k\left(\frac{L}{4}+\frac{k}{1+k}\left(z-\frac{L}{4}\right)\right)
\end{equation}
for $0\leq z\leq L/2$ (such a sequence can be constructed from mollifications of $\psi$).  Let $v_k$ be the regular solution of
\begin{align*}
\frac{\partial v_k}{\partial t}(z,t) &= \alpha(v_k'(z,t))v_k''(z,t),\quad (z,t)\in[0,L/2]\times[0,\infty);\\
v_k(z,0)&=\psi_k(z),\quad z\in[0,L/2];\quad
v_k(0,t)=v_k(L/2,t)=0,\quad t\geq 0.
\end{align*}
(see Appendix A for the regularity results required to prove the existence of $v_k$)
\begin{lemma}
For each $k$ and each $t>0$, $v_k(.,t)$ is concave.
\end{lemma}

\begin{proof} 
The time-independent function $\psi_k(z)$ is a supersolution of \eqref{eq:1Dflow} with zero boundary data, so by the comparison principle, $v_k(z,t)\leq \psi_k(z)$ for all $t\geq 0$.  Applying the comparison principle again, we find that the solution of \eqref{eq:1Dflow} with initial data $v_k(z,s)$ is less than or equal to the solution with initial data $\psi_k(z)=v_k(z,0)$ for all positive times, i.e. $v_k(z,s+t)\leq v_k(z,t)$ for all $t\geq 0$ and $s\geq 0$.  Since $v_k$ is regular, $\frac{\partial v_k}{\partial t}$ exists, and $\frac{\partial}{\partial t}v_k(z,t)=\lim_{s\to 0_+}s^{-1}\left(v_k(z,t+s)-v_k(z,t)\right)\leq 0$ for all $z$ and all $t\geq 0$.  Since \eqref{eq:1Dflow} is satisfied and $\alpha$ is positive, $v_k''(z,t)\leq 0$.
\end{proof}

\begin{lemma}\label{lem:oscvk}
For any small $\delta>0$ there exists $C_\delta(\alpha, \sup\psi)$ independent of $k$ such that 
$$
\sup_{(z,y,t)\in[\delta,L/2-\delta]\times[0,L/2]\times[0,\infty)}\frac{|v_k(y,t)-v_k(z,t)|}{|y-z|} + \sup_{z\in[\delta,L/2-\delta],\ 0\leq t_1< t_2}
\left|\frac{v_k(z,t_2)-v_k(z,t_1)}{\sqrt{t_2-t_1}}\right|\leq C_\delta
$$
and
$$
\sup_{[\delta,L/2-\delta]\times[0,\infty)}\left|v_k'(z,t)\right|+\sup_{[\delta,L/2-\delta]\times[\delta,\infty)}\left(\left|v_k''\right|+\left|\frac{\partial v_k}{\partial t}\right|\right)\leq C_\delta.
$$
Furthermore, $\{v_k'':\ k\in\NN\}$ and $\left\{\frac{\partial v_k}{\partial t}:\ k\in\NN\right\}$ are equicontinuous on $[\delta,L/2-\delta]\times[\delta,\infty)$.
\end{lemma}

\begin{proof}
Since $v_k(.,t)$ is concave, we have for $z\in[\delta,L/2-\delta]$, $y\neq z$ in $[0,L/2]$ and $t\geq 0$:
\begin{equation}\label{eq:xoscest}
-\frac{\sup\psi}{\frac{L}{2}-\delta}\leq\frac{v_k(\frac{L}2,t)-v_k(z,t)}{\frac{L}2-z}\leq \frac{v_k(z,t)-v_k(y,t)}{z-y}\leq \frac{v_k(z,t)-v_k(0,t)}{z}\leq \frac{\sup\psi}{\delta}.
\end{equation}
Let $\tilde C_\delta=\max\left\{\frac{\sup\psi}{\delta},\frac{\sup\psi}{L/2-\delta}\right\}$.  Taking $y\to z$ in \eqref{eq:xoscest} gives $|v_k'(z,t)|\leq \tilde C_\delta$ for $(z,t)\in[\delta,L/2-\delta]\times[0,\infty)$.
Also, \eqref{eq:xoscest} implies
$v_k(y,t)\geq v_k(z,t)-\tilde C_\delta|y-z|$
for any $y\in[0,L/2]$, $z\in[\delta,L/2-\delta]$ and $t\geq0$.   Define $\alpha_\delta = \sup_{|p|\leq\tilde C_\delta}\alpha(p)$.  Fix $z$ and define $w(y,s)=v_k(z,t)-\tilde C_\delta\sqrt{s\alpha_\delta}\,\mu\left(\frac{y-z}{\sqrt{s\alpha_\delta}}\right)$, where $\mu(\xi) = \frac{2}{\sqrt{\pi}}\exp\left\{-\frac{\xi^2}{4}\right\}+\xi\text{erf}\left(\frac{\xi}{2}\right)$ (chosen such that $\sqrt{t}\mu(x/\sqrt{t})$ solves the heat equation with initial data $|x|$). Note $|w'(y,s)|\leq \tilde C_\delta$, $w''\leq 0$, and
$\frac{\partial w}{\partial s} = \alpha_\delta w''\leq \alpha(w')w''$,
so $w$ is a subsolution of \eqref{eq:1Dflow}.  By the comparison principle $v_k(y,t+s)\geq w(y,s)$ for all $y$ and all $s>0$, and in particular taking $y=z$ we find
$$
v_k(z,t)\geq v_k(z,t+s)\geq v_k(z,t)-\frac{2\tilde C_\delta\sqrt{s\alpha_\delta}}{\sqrt{\pi}}.
$$
Therefore
\begin{equation}\label{eq:tcontest}
\sup_{z\in[\delta,L/2-\delta],\ 0\leq t_1\leq t_2}\frac{\left|v_k(z,t_2)-v_k(z,t_1)\right|}{\sqrt{t_2-t_1}}\leq \frac{2\tilde C_\delta\sqrt{\alpha_\delta}}{\sqrt{\pi}}.
\end{equation}
It follows from the bound on $v_k'$ that $v_k$ is a solution of a uniformly parabolic equation on $[\delta,L/2-\delta]\times[0,\infty)$, and Equation \eqref{eq:boundw2} applies to give $k$-independent bounds on $v_k''$ and $\frac{\partial v_k}{\partial t}$ on $[2\delta,L/2-2\delta]\times[2\delta,\infty)$, and \eqref{eq:Holderw} and \eqref{eq:ctsd2u} control their moduli of continuity.
\end{proof}

Define $\varphi_k(z,t) = \frac{1+k}{k}v_k\left(\frac{L}{4}+\frac{k}{1+k}\left(z-\frac{L}{4}\right),\left(\frac{k}{1+k}\right)^2t\right)$.  Then $\varphi_k\in{\mathcal S}_\psi$, so $\psi_+(z,t)\leq \varphi_k(z,t)$ for every $k$.  For any $\varphi\in{\mathcal S}_\psi$ the comparison principle implies $\varphi\geq v_k$, and hence $v_k(z,t)\leq \psi_+(z,t)\leq \varphi_k(z,t)$.  
By Lemma \ref{lem:oscvk} we have for $\delta\leq z\leq L/2-\delta$ and $t\geq 0$
\begin{align*}
\varphi_k(z,t)-v_k(z,t)&=\frac{1+k}{k}v_k\!\left(\!\frac{L}{4}+\frac{k}{1+k}\!\left(\!z-\frac{L}4\right),\left(\frac{k}{1+k}\right)^2\!t\right)-v_k(z,t)\\
&\leq\frac1k\sup v_k + \left|v_k\!\left(\!\frac{L}{4}\!+\!\frac{k}{1+k}\!\left(\!z-\frac{L}4\right),\left(\frac{k}{1+k}\right)^2\!t\right)
\!-v_k\!\left(\!z,\left(\frac{k}{1+k}\right)^2\!t\right)\right|\\
&\quad\null + \left|v_k\!\left(z,\left(\frac{k}{1+k}\right)^2t\right)-v_k(z,t)\right|\\
&\leq \frac1k\sup\psi + \frac1{1+k}\left|z-\frac{L}4\right|\tilde C_\delta + \frac{2\tilde C_\delta\sqrt{\alpha_\delta(1+2k)}}{(1+k)\sqrt{\pi t}}\\
&\leq C\left(\frac{1}{k}+\sqrt{\frac{t}{k}}\right),
\end{align*}
so that $v_k$ and $\varphi_k$ both converge locally uniformly to $\psi_+$ on $[\delta,L/2-\delta]\times[0,\infty)$.
Therefore $\psi_+$ is locally a uniform limit of concave functions, hence concave.  Furthermore, the estimates on $v_k''$ and $\frac{\partial v_k}{\partial t}$ and their moduli of continuity imply that $\psi_+$ is a regular solution of \eqref{eq:1Dflow} on $(0,L/2)\times(0,\infty)$.
\end{proof}
 
\begin{corollary}\label{cor:optimalbound}
Every $L$-periodic solution $u$ of \eqref{eq:1Dflow} for which $\psi$ is a modulus of continuity for $u(.,0)$ has $\psi_+(.,t)$ as a modulus of continuity for $u(.,t)$ each $t>0$. This estimate is sharp:  For any $(z,t)\in (0,L/2)\times(0,\infty)$ and any $\rho<\psi_+(z,t)$ there exists an $L$-periodic regular solution $u$ of \eqref{eq:1Dflow} such that $\psi$ is a modulus of continuity for $u(.,0)$ but there exists $x\in\RR$ with $|u(x+2z,t)-u(x,t)|>2\rho$.
\end{corollary} 

\begin{proof}
The estimate was established in the remarks after Definition \ref{def:minsuper}.  The sharpness follows from Theorem \ref{thm:sharpbarrier}:  The functions $v_k$ extend by odd reflections to $L$-periodic regular solutions of \eqref{eq:1Dflow}, and for $k$ sufficiently large $v_k(z,t)-v_k(-z,t)=2v_k(z,t)>2\rho$.  It remains to check that $v_k(.,0)$ has modulus of continuity $\psi$.  We show more generally that if $v$ is a function which is concave and positive on $[0,L/2]$, $L$-periodic and odd  then $v$ has modulus of continuity $v$ (hence also modulus of continuity $\psi$ if $v\leq\psi$ on $(0,L/2)$):  Consider any $x,y$ with $0<y-x<L$.

\noindent{\bf Case 1:} $nL\leq x<\left(n+\frac12\right)L$ for some $n\in\ZZ$.  By periodicity we can assume $n=0$.  Then we have three sub-cases:

If $x<y\leq \left(n+\frac12\right)L$, then by concavity
$v(y-x)\geq v(y)-v(x)\geq -v\left(\frac{L}2+x-y\right)$, and also by concavity $2v\left(\frac{y-x}{2}\right)\geq v(y-x)$ and $2v\!\left(\!\frac{L+x-y}{2}\!\right)\!\geq\! v\!\left(\frac{L}2\!+\!x\!-\!y\right)\!+\!v\!\left(\frac{L}2\right)\!=\!v\!\left(\frac{L}2\!+\!x\!-\!y\right)$.

If $x<\left(n+\frac12\right)L< y \leq (n+1)L$, then $2v\left(\frac{y-x}{2}\right)\geq 0\geq v(y)-v(x)=-v(x)-v(L-y)\geq -2v\left(\frac{L+x-y}{2}\right)$.

The remaining possibility is $\left(n+\frac12\right)L<y<x+L$.  Then $v(y)-v(x) = v(y-L)-v(x)$, so by the first sub-case above,
$$
-2v\!\left(\frac{L+x-y}{2}\right)\!=\!-2v\!\left(\frac{x-y+L}{2}\right)\leq v(y)-v(x)\leq 2v\!\left(\frac{L\!+\!y\!-\!L\!-\!x}{2}\right)\!=\!2v\!\left(\frac{y-x}{2}\right).
$$

\noindent{\bf Case 2:} $\left(n+\frac12\right)L\leq x<(n+1)L$  (as before assume $n=0$).  Then $v(y)-v(x)=v(L-x)-v(L-y)$, so by Case 1,
$$
-2v\left(\frac{L+(L-y)-(L-x)}{2}\right)\leq v(y)-v(x)\leq 2v\left(\frac{L-x-L+y}{2}\right),
$$
which is the required result after rearrangement.
\end{proof}

\section{Translating solutions}\label{sec:trans}

This section concerns translating solutions of Equation \eqref{eq:1Dflow}.  These will be used in the next section to obtain estimates on the minimal supersolution $\psi_+$.

Translating solutions are special solutions of Equation \eqref{eq:1Dflow} of the form
$$
v(x,t) = v(x,0)-Vt
$$ for some constant $V$.  These satisfy the ordinary differential
equation
\begin{equation}\label{eq:translator}
\alpha(v')v'' = -V.
\end{equation}
The solutions of \eqref{eq:translator} can be obtained by integration:  Define $A(\xi) = \int_{0}^{\xi}\alpha(s)\,ds$ and
$B(\xi) = \int_{0}^{\xi}s\alpha(s)\,ds$.  Then Equation \eqref{eq:translator} implies
$$
A(v'(s))-A(v'(s_{0})) = -V(s-s_{0})
$$
and
$$
B(v'(s))-B(v'(0)) = -V(v(s)-v(s_{0})).
$$
Combining these gives the expression
\begin{equation}\label{eq:transsoln}
v(s) = v(s_0) -\frac{1}{V}\left(B\circ A^{-1}\left(A(v'(s_0))-V(s-s_0)\right)-B(v'(s_0))\right).
\end{equation}
Alternatively, the graph of the translating solution can be described parametrically:
\begin{equation}\label{eq:transparam}
\tx{graph}(v) = \left\{(x_0,y_0)+\frac{1}{V}\left(A(p),B(p)\right):\ p\in\RR\right\}.
\end{equation}
The parameter $p$ then gives the slope of the graph at the corresponding point.

Note that all translating solutions are related by translation (horizontally and/or vertically) and scaling.  Also note that
$\lim_{\xi\to\infty}A(\xi) = \infty$ if and only if the domain of definition of the translating solution extends to $s=-\infty$, and
$\lim_{\xi\to-\infty}A(\xi) = -\infty$ if and only if the domain of definition extends to $s=\infty$.  If $\lim_{\xi\to\infty}A(\xi) = \infty$ then
$\lim_{\xi\to\infty}B(\xi)=\infty$, and if $\lim_{\xi\to-\infty}A(\xi) = -\infty$ then $\lim_{\xi\to-\infty}B(\xi)=\infty$, and so in these cases
the translating solutions extend to $-\infty$ in the vertical direction as $s\to\infty$ or $s\to-\infty$ respectively.  However there can
be cases where the range of $A$ is finite but the range of $B$ is infinite, so that the function $v$ approaches $-\infty$ at a finite value of $s$, or where the ranges of both $A$ and $B$ are finite, so that the gradient of the function becomes infinite with finite values of both $s$ and $v$.

\begin{example}\label{exa:heateq1}
For the heat equation, $\alpha=1$ everywhere, so $A(\xi) = \xi$ and $B(\xi) = \frac12\xi^{2}$.  The translating solutions are therefore
parabolae $\left\{\frac{1}{V}(\xi,\frac12\xi^{2})\right\}=\left\{(x,\frac{V}{2}x^{2})\right\}$, which extend to infinity both horizontally and vertically.
\end{example}

\begin{example}
For the curve-shortening flow, $\alpha(p)=\frac{1}{1+p^{2}}$, so 
$$
A(\xi) = \arctan(\xi)
$$
and
$$
B(\xi) = \log\sqrt{1+\xi^{2}}.
$$
Thus the translating solutions have finite width but extend to $-\infty$ at both ends in the vertical direction.  Note that since
$\xi=\tan A=\tan(sV)$, we have (up to translations)
$$
v(s) = -\frac{1}{V}\log\cos(Vs).
$$
\end{example}

\begin{example} In the homogeneous case, $\alpha(p)=|p|^{-\gamma}$, so such $\alpha$ are \emph{not} $C^0(\RR)$.  Nevertheless,  we can find  translating solutions $v(s,t)=v(s)-Vt$, with
\begin{equation*}
v(s)=\begin{cases} -\frac{1}V \frac{|sV(1-\gamma)|^{\frac{2-\gamma}{1-\gamma}}}{2-\gamma},
&\gamma\not=\lbrace 1, 2\rbrace \\
-\frac{1}V \exp (\pm V s), & \gamma=1 \\
-\frac{1}V \log |sV|, & \gamma=2. 
\end{cases}
\end{equation*}
The nature of the solution is determined by the exponent:
\begin{equation*}
\frac{2-\gamma}{1-\gamma}\in 
\begin{cases}
(2,\infty) & \text{for } 0<\gamma<1 \\
(-\infty,0) & \text{for } 1<\gamma<2 \\
(0,1) & \text{for } \gamma>2,
\end{cases}\end{equation*}
with $v$ being
\begin{align*}
&\text{unbounded horizontally and vertically } & 0<\gamma\le 1, \\
&\text{unbounded vertically, but with a one-sided bound horizonally } & 1<\gamma\le 2, \\
&\text{bounded on one side both vertically and horizontally } & \gamma>2.
\end{align*}
\end{example}

\begin{example}  \label{exa:asymhom} We can also find translating solutions in  the \emph{asymptotically}  homogeneous case, $\alpha(p)\sim |p|^{-\gamma}$ as $|p|\rightarrow \infty$.  Here,
$\lim_{\xi\to\infty}A(\xi) = \infty$ for $\gamma\leq 1$, and $\lim_{\xi\to\infty}A(\xi)<\infty$ for $\gamma>1$, while
$\lim_{\xi\to\infty}B(\xi) = \infty$ for $\gamma\leq 2$, and $\lim_{\xi\to\infty}B(\xi)<\infty$ for $\gamma>2$.  Thus translating solutions 
are unbounded both horizontally and vertically if $\gamma\geq 1$, vertically but not horizontally if $1<\gamma\leq 2$, and neither vertically nor horizontally if $\gamma>2$.  The curve shortening flow is an example of  the critical case $\gamma=2$, while the non-parametric curve-shortening flow 
$$
\frac{\partial u}{\partial t} = \frac{u''}{\left(1+(u')^{2}\right)^{3/2}}
$$
(where the curve moves vertically with speed equal to the curvature) corresponds to $\gamma=3$.  In the latter example the translating solution is (up to scaling and translations)
$$
v(s) = -\sqrt{1-s^{2}},
$$
(semicircles have constant curvature and hence constant speed in this flow).  

More generally for asymptotically homogeneous cases, the translating solutions are asymptotic to those for the homogeneous cases:
\begin{equation*}
v(s)\rightarrow  
\begin{cases}
-  \dfrac{1}{V(2-\gamma)}\left|V(1-\gamma)s\right|^ {\frac{2-\gamma}{1-\gamma}} \text{ as } |s|\rightarrow \infty,& 0<\gamma<1 \\
-\dfrac 1 V \exp (V|s|) \text{ as } |s|\rightarrow \infty,& \gamma=1 \\
-\dfrac 1 {V(2-\gamma)}\left|V(1-\gamma)(s^\pm -s)\right|^{\frac{2-\gamma}{1-\gamma}} \text{ as } s\rightarrow s^\pm, & 1<\gamma<2 \text{ and } \gamma>2 \\
\dfrac 1 V \log \left|V(s-s^\pm)\right| \text{ as } s\rightarrow s^\pm, & \gamma=2
\end{cases}\end{equation*}
where $s^-=-V^{-1}\int_0^{\infty}\alpha(s)\,ds$ and $s^+=V^{-1}\int_{-\infty}^0\alpha(s)\,ds$. 
\end{example}

The dichotomy between finite and infinite vertical extent for translating solutions corresponds precisely to the sufficient condition for
Lipschitz bounds for bounded initial data which we provide in Corollary \ref{cor:bddbarrier} in the following section.

\section{Necessary and sufficient condition for Lipschitz bound}\label{sec:suff}

In this section we use supersolutions constructed from translating solutions to
give a sufficient condition for the minimal supersolution $\psi_+$ to have bounded gradient for positive times.  In many cases of interest this is also a necessary condition, though this is not always true.  By Corollary \ref{cor:optimalbound} these results amounts to necessary and sufficient conditions for a given coefficient $\alpha$ and initial modulus of continuity $\psi$ to imply gradient bounds for positive times.

\begin{theorem}\label{thm:barriercriterion}
Let $\psi$ be positive and concave on $(0,L/2)$, and let $\psi_+$ be the minimal supersolution of Definition \ref{def:minsuper}.  Let $b$ be minus the Legendre transform of $\psi$, so that 
\begin{equation*}
\psi(x) = \inf\left\{x\cdot z +b(z):\ z\in{\mathbb R}\right\}
\end{equation*}
and
\begin{equation*}
b(z) = \sup\left\{ \psi(x) - x\cdot z:\ 0\leq x\leq L\right\}.
\end{equation*}
Also define $\tilde b(z) = b(z)+\frac{Lz}{2}$.
Then $\psi_+(.,t)$ has gradient bounded above for each $t>0$ if 
\begin{equation}\label{eq:1Dcriterion}
\inf_{z\in\RR}
\frac{b(z)}{\left(\lim_{z'\to\infty}\int_z^{z'} (s-z)\alpha(s)\,ds\right)^{1/2}} = 0.
\end{equation}
Furthermore, if this holds then we have the estimate
\begin{equation}\label{eq:gradest}
\psi_+'(x,t) \leq \inf\left\{Z:\ \int_z^{Z}(s-z)\alpha(s)\,ds \geq \frac{b(z)^2}{t}\tx{ for some }z\right\}
\end{equation}
for all $x\in[0,L/2]$. Similarly, $\psi_+(.,t)$ has gradient bounded below at for each $t>0$ if
\begin{equation}
\inf_{z\in\RR}
\frac{\tilde b(z)}{\left(\lim_{z'\to-\infty}\int_{z'}{z} (z-s)\alpha(s)\,ds\right)^{1/2}}=0.
\end{equation}
In this case the following estimate holds for each $x\in[0,L/2]$:
\begin{equation}\label{eq:rightgradest}
\psi_+'\left(x,t\right)\geq \sup\left\{Z:\ \int_Z^z(z-s)\alpha(s)\,ds \geq \frac{\tilde b(z)^2}{t}\tx{ for some }z\right\}.
\end{equation}
\end{theorem}

\begin{proof}
Suppose \eqref{eq:1Dcriterion} holds, and 
fix $t_0>0$.  Then in particular there exists $z$ such that $\int_z^\infty (s-z)\alpha(s)\,ds >\frac{2b(z)^2}{t_0}$, and hence there also exists $Z>z$ such that $\int_z^Z (s-z)\alpha(s)\,ds = \frac{2b(z)^2}{t_0}$.  We prove that $\psi_+'(0,t_0)\leq Z$ for any such $z$ and $Z$.

Let $v$ be the translating solution with graph $\{(x(p),y(p,t)):\ z\leq p\leq Z, 0\leq t\leq t_0\}$ defined parametrically by \eqref{eq:transparam}, with $x(Z)=0$, $y(Z,0)=b(z)$, and $V=\frac{b(z)}{t_0}$.  Therefore $y(Z,t)=y(Z,0)-tV\in [0,b(z)]$ for $0\leq t\leq t_0$.  The graph has slope $z\leq p\leq Z$, hence $v(x,0)\geq b(z)+zx$.  It follows that $v(x,0)\geq\psi(x)$ since
$\psi(x)-xz\leq b(z)$ for every $x$, so $\psi(x)\leq b(z)+xz\leq v(x,0)$ for each $x\in [0,x(z)]$.
Finally, we compute 
\begin{align*}
y(Z,t) -zx(Z) &= \left(b(z)+\frac{1}{V}\int_z^Z r\alpha(r)\,dr-Vt\right)-\frac{z}{V}\int_z^Z \alpha(r)\,dr\\
&=b(z)+\frac{t_0}{b(z)}\int_z^Z(r-z)\alpha(r)\,dr-\frac{b(z)t}{t_0}\\
&=2b(z)-\frac{t}{t_0}b(z)\\
&\geq b(z)
\end{align*}
and therefore we have $y(Z,t)\geq x(Z)z+b(z)\geq \psi(x(Z))\geq \psi_+(x(Z),t)$ for $0\leq t\leq t_0$.  By the comparison principle, $\psi_+(x,t)\leq v(x,t)$ for $0\leq x\leq x(Z)$ and $0\leq t\leq t_0$, and in particular $\psi_+(x,t_0)\leq v(x,t_0)\leq Zx$ for all $x$, so $\psi_+'(0,t_0)\leq Z$.  By concavity of $\psi_+(.,t)$ this implies $\psi_+'(x,t)\leq Z$ for every $x$.  The other inequality is proved similarly.
\end{proof}

\begin{remark}
The condition \eqref{eq:1Dcriterion} is not in general necessary for a gradient bound.  Under some circumstances this is the case, for example if there exists a convex function $\mu$ with $\lim_{z\to\infty}\mu=0$ such that $\alpha(z) = (\mu^q)_{zz}$ for some $q>2$.  In this case it can be shown that there exists a homothetic solution $v$ of the form 
$$
v(x,t) = \sqrt{-t}\,v_0\left(\frac{x}{\sqrt{-t}}\right)
$$
for $0<x<C\sqrt{-t}$ for some $C$,
with $v_0(0)=0$ and $v_0'(0)=\infty$.  Furthermore $v_0$ has Legendre transform comparable to
$\left(\int_z^\infty(s-z)\alpha(s)\,ds\right)^\frac12$ for large $z$.  It follows that any $\psi$ which does not satisfy condition \eqref{eq:1Dcriterion} lies above $v(.,t_0)$ for some $t_0<0$.  A comparison principle argument then shows that the minimal supersolution $\varphi(.,t)$ lies above $v(.,t_0+t)$ for small $t$, and so cannot have bounded gradient for these times.

It seems a plausible conjecture that whenever $B$ is bounded, there exists a homothetic solution of this kind, and that this determines the optimal modulus of continuity.  Examples show that in some cases the resulting modulus of continuity is not comparable to that given in Condition \eqref{eq:1Dcriterion}, so the optimal condition on initial data is a weaker one.
\end{remark}

\begin{corollary}\label{cor:bddbarrier}
If $\psi (z) =M$ for $0<z<L/2$, then $\psi_+$ has gradient bounded above for each $t>0$ if and only if $\lim_{\xi\to\infty}B(\xi)=\infty$,  and gradient bounded below for each $t>0$ if and only if $\lim_{\xi\to-\infty}B(\xi)=\infty$.   In this case 
$$
\inf\left\{\xi:\ B(\xi)\leq \frac{M^2}{t}\right\}\leq \psi_+'(z,t)\leq \sup\left\{\xi:\ B(\xi)\leq \frac{M^2}{t}\right\}.
$$
\end{corollary}

\begin{proof}  In this case we have $b(0)=\tilde b(0)=M$, and the gradient bound follows by setting $z=0$ in \eqref{eq:gradest} and \eqref{eq:rightgradest}.

On the other hand, suppose that $\lim_{\xi\to\infty}B(\xi)<\infty$.  Then $\lim_{\xi\to\infty}A(\xi)<\infty$ also. 
Choose $V>\max\left\{\frac{4}{L}\lim_{\xi\to\infty}A(\xi),\frac{4}{M}\lim_{\xi\to\infty}B(\xi)\right\}$.   Let $X=\frac{1}{V}\lim_{\xi\to\infty}A(\xi)<\frac{L}{4}$.  The estimate \eqref{eq:tcontest} applies to $\psi_+$, since $\psi_+$ is locally the uniform limit of the functions $v_k$.  Therefore there exists $\tau>0$ such that
$\psi_+(x,t)\geq \frac{3M}{4}$ for $0\leq t\leq \tau$ and $X\leq x\leq \frac{L}{2}-X$.

Let $\varepsilon>0$ be sufficiently small to ensure that $X+\varepsilon\leq\frac{L}{2}$.  Let $v$ be the translating solution with speed $V$ defined parametrically by $x(p)=\varepsilon+\frac{1}{V}\int_p^\infty \alpha(s)\,ds$ and $v(p,t) = \frac{M}{2}+\frac{1}{V}\int_p^\infty s\alpha(s)\,ds-Vt$, with $0\leq t\leq \tilde\tau=
\min\left\{\tau,\frac{M}{4V}\right\}$.  For each $p$ we have $v(p,t)\geq \frac{M}{2}-Vt\geq \frac{M}{2}-V\frac{M}{4V} = \frac{M}{4}$, and we have $v(p,0)\leq \frac{M}{2}+\frac{1}{V}\lim_{\xi\to\infty}B(\xi)\leq \frac{3M}{4}\leq \psi(x(p))$, and for each $t$ in this range we have $v(0,t)\leq \frac{3M}{4}\leq \psi_+(x(0),t)$.  Finally, since $\lim_{p\to\infty}v'(x(p),t)=\infty$, the minimum of $\psi_+(.,t)-v(.,t)$ cannot occur at the left endpoint.  Therefore by the comparison principle we have $\psi_+(x,t)\geq v(x,t)$ for $\varepsilon\leq x\leq X+\varepsilon$ and $0\leq t \leq\tilde\tau$.  This holds for every $\varepsilon>0$, and therefore 
$\psi_+(0,t)\geq \frac{M}{4}$ for $0\leq t\leq \tilde\tau$ and $\psi_+$ does not have bounded gradient on this time interval.  The argument for the failure of the lower gradient bound is similar.
\end{proof}

Corollary \ref{cor:bddbarrier} and Theorem \ref{thm:sharpbarrier} give a necessary and sufficient condition for the Lipschitz constants of regular solutions to be controlled by initial oscillation and elapsed time:
\begin{corollary}\label{cor:lipboundcriterion}
$L$-periodic solutions of Equation \eqref{eq:1Dflow} satisfy estimates
of the form
$$
|u'(z,t)|\leq C(t,|u|_{\infty})
$$
for each $t>0$
if and only if $B$ approaches infinity at $\pm\infty$.
\end{corollary}

\section{Examples:  Gradient-dependent coefficients}\label{sec:examples}

In this section we explore the application of the previous results to some examples of equations with
gradient-dependent coefficients.  

\begin{example}\label{exa:heateq2}
We first consider the case of the heat equation, $\alpha(p)=1$ for all $p$.  As observed in Example \ref{exa:heateq1}, the
translating solutions are unbounded, and so Lipschitz bounds hold for positive times, by Corollary \ref{cor:lipboundcriterion}.  The estimate provided by Theorem \ref{thm:barriercriterion} for bounded solutions ($b(z)=M$ for all $z$) is then
$$
|\psi_+'(z,t)|\leq \frac{\sqrt{2}M}{\sqrt{t}}.
$$
(In this case the sharp estimate is $\left|\psi'(z,t)\right|\leq M\left(4\pi t\right)^{-1/2}$).
We can also deduce a stronger estimate if the initial data is H\"older continuous:  If $\psi(z) = Cz^{\beta}$, then we have
$b(z) = C^{1/(1-\beta)}\beta^{\beta/(1-\beta)}(1-\beta)z^{-\beta/(1-\beta)}$, so the gradient estimate becomes
$$
|\psi_+'(z,t)|\leq CC_2t^{-\frac{1-\beta}{2}}
$$
where $C_2=C_2(\beta)$.
\end{example}

\begin{example}\label{exa:asymhom2}
Consider the asymptotically homogeneous classes discussed in Example \ref{exa:asymhom}, where $\alpha(p)\sim|p|^{-\gamma}$.     For bounded initial data ($\psi(z)=M$) condition \eqref{eq:1Dcriterion} is satisfied only when $\gamma\le 2$, and then the gradient estimate given by Theorem \ref{thm:barriercriterion} for small $t$ is  
\begin{equation*}
|\psi_+'(z,t)|\le
\begin{cases} 
 \left(
\dfrac{CM ^2}{t}
\right)^{\frac{1}{2-\gamma} }, & \gamma<2 \\
\exp\left(\dfrac{CM^2}t\right) & \gamma=2.  
\end{cases}
\end{equation*}

If the initial data is H\"older continuous, with $\psi(z)=Kz^\beta$ and $b(z)$
as in the above example, then condition \eqref{eq:1Dcriterion} is satisfied when $\gamma<\frac{2}{1-\beta}$.     The gradient estimate is  
\begin{equation*}
|\psi_+'(z,t)|\le \tilde{C}K^{\frac{2}{2-(1-\beta)\gamma}}t^{-\frac{1-\beta}{2-(1-\beta)\gamma}}.
\end{equation*}
This cannot be improved to $\gamma=\frac{2}{1-\beta}$ if $\gamma>2$, since the function $y=x^{\frac{\gamma-2}{\gamma}}$ gives $\alpha(y')y''\propto -y$, so there is a self-similar subsolution with this H\"older exponent, and consequently $\psi_+(.,t)$ does not have bounded gradient for small $t>0$ in such cases.
\end{example}

\section{Coefficients depending on position, time and height}\label{sec:coeff}

The methods introduced above for equations with coefficients depending only on the gradient also apply much more generally, provided the coefficients can be estimated from below by a function of the gradient.

Consider regular solutions of equations of the form
\begin{equation}\label{eq:1Dgenflow}
\frac{\partial u}{\partial t}(x,t) = \tilde\alpha(x,t,u(x,t),u'(x,t))u''(x,t)
\end{equation}
where $\tilde a(x,t,z,p)\geq \alpha(p)$.  We require $\tilde\alpha$ to be $L$-periodic in $x$.  The main result is as follows:

\begin{theorem}
Let $u$ be a regular $L$-periodic solution of \eqref{eq:1Dgenflow} on $\RR\times[0,T]$ with initial modulus of continuity $\psi$, where $\psi$ is positive and concave on $(0,\frac{L}{2})$.  Then $u(.,t)$ has modulus of continuity $\psi_+(.,t)$ for each $t>0$, where $\psi_+$ is defined as in Definition \ref{def:minsuper} using supersolutions of \eqref{eq:1Dflow}.
\end{theorem}

In particular, this implies that the conditions in Theorems \ref{thm:barriercriterion} and \ref{cor:bddbarrier}  are sufficient for Lipschitz bounds for such equations.

\begin{proof}
We follow the proof of Theorem \ref{thm:1Dgeneral}:  Defining
$Z(x,y,t)=u(y,t)-u(x,t)-2\varphi\left(\frac{y-x}{2},t\right)-\varepsilon(1+t)$ as before, we arrive at the folowing at the first point and time $(x_0,y_0,t_0)$ where $Z$ reaches zero:
\begin{align*}
0\leq \frac{\partial Z}{\partial t}(x,y,t)
&=\tilde\alpha\left(y,t,u(y,t),u'(y,t)\right)u''(y,t)-\tilde\alpha\left(x,t,u(x,t),u'(x,t)\right)u''(x,t)\\
&\quad\null-2\frac{\partial\varphi}{\partial t}\left(\frac{y-x}{2},t\right)-\varepsilon\\
&<\tilde\alpha(y,t,u(y,t),\varphi')\left(\frac{\partial^2Z}{\partial y^2}+\frac{\varphi''}{2}\right)
+\tilde\alpha(x,t,u(x,t),\varphi')\left(\frac{\partial^2Z}{\partial x^2}+\frac{\varphi''}{2}\right)\\
&\quad\null +2c\left(\frac{\partial^2Z}{\partial y\partial x}-\frac{\varphi''}{2}\right)
-2\frac{\partial\varphi}{\partial t}.
\end{align*}
At this point $\frac{\partial^2Z}{\partial y^2}\leq 0$, and since $\varphi$ is concave we have
$\frac{\partial^2Z}{\partial y^2}+\frac{\varphi''}{2}\leq 0$.  Therefore 
$$
\tilde\alpha\left(y,t,u(y,t),\varphi'\right)\left(\frac{\partial^2Z}{\partial y^2}+\frac{\varphi''}{2}\right)
\leq \alpha(\varphi')\left(\frac{\partial^2Z}{\partial y^2}+\frac{\varphi''}{2}\right),
$$
and similarly for the terms involving $x$.  Thus exactly as in Theorem \ref{thm:1Dgeneral} we find
$$
0< \text{\rm Tr}
\left(
\bmatrix \alpha(\varphi')&c\cr c&\alpha(\varphi')\cr\endbmatrix
\bmatrix \frac{\partial^2Z}{\partial x^2}&\frac{\partial^2Z}{\partial x\partial y}\cr
\frac{\partial^2Z}{\partial y\partial x}&\frac{\partial^2Z}{\partial y^2}\cr\endbmatrix\right)+\left(\alpha(\varphi')-c\right)\varphi''-2\frac{\partial\varphi}{\partial
t},
$$
yielding the required contradiction with the choice $c=-\alpha(\varphi')$.
\end{proof}

\appendix
\section{Regularity results for uniformly parabolic flows}\label{sec:unifexist}

This section concerns regularity results for quasilinear parabolic equations in which the coefficients depend only on the gradient of the solution.  
In particular we will deduce estimates on continuity for the second derivatives of solutions to uniformly parabolic equations, assuming only that the coefficients are continuous.   The  results are not crucial for the main results of the paper, but they make possible results which assume only continuity of the coefficient $\alpha$.

The corresponding estimates fail for more general parabolic equations, even in one dimension:  A simple example is provided by the function
$$
u(x,t) = (x^2+t)\log\left(1+\frac{1}{x^2-t}\right)+C_3t+C_4 x^2
$$
on $\RR\times(-\infty,0]$ (extending to value zero at $(0,0)$).  For $C_3$ and $C_4$ large this is
a solution of an equation of the form
$$
\frac{\partial u}{\partial t}(x,t) = \alpha(x,t)\frac{\partial^2 u}{\partial x^2}(x,t)
$$
with the coefficients $\alpha(x,t)$ continuous on $\RR\times(-\infty,0]$ and satisfying $0<\lambda\leq\alpha(x,t)\leq \Lambda$ for some constants $\lambda$ and $\Lambda$ (a similar example appears in \cite{il1967parabolic}).  However $u$ does not have continuous or even bounded second spatial derivatives near $(0,0)$.  Schauder estimates provide continuity of the second derivatives if the coefficient $\alpha(x,t)$ is Dini-continuous \cite{MatEid}.

Now we proceed to the estimates:
Let $\alpha$ be continuous on $\RR$, with 
\begin{equation}\label{eq:ellipt}
0<\lambda\leq \alpha(p)\leq\Lambda
\end{equation}
for all $p$, for some constants $\lambda$ and $\Lambda$, and
\begin{equation}\label{eq:alphacts}
|\alpha(p)-\alpha(q)|\leq \omega(|p-q|)
\end{equation} 
for all $p$ and $q$, where $\omega$ is continuous on $[0,\infty)$ with $\omega(0)=0$.
Denote $Q_R(0,0)=(-R,R)\times(-R^2,0]$.  Consider a solution $u$ on $Q_R(0,0)$ of the equation
\begin{equation}\label{eq:graddepflow}
\frac{\partial u}{\partial t} = \alpha(u')u''
\end{equation}
with $\tx{osc}(u)\leq M$.  The argument of Kruzhkov outlined in the introduction gives 
$$
|u'(x,t)|\leq CM\left(\frac{\Lambda}{\lambda d^{2}}+\frac{1}{\lambda (t+R^2)}\right)^{1/2},
$$
where $d=R-|x|$ is the distance to the boundary.  In particular on $Q_{R/2}(0,0)$ we have
\begin{equation}\label{eq:gradesthalfR}
|u'(x,t)|\leq C\left(\frac{1+\Lambda}{\lambda}\right)^{1/2}\frac{M}{R}.
\end{equation}
Equation \eqref{eq:graddepflow} implies $v=u'$ is a weak solution on $Q_{R/2}$ of the
divergence-form equation
\begin{equation}\label{eq:evolderiv}
\partial_tv = \partial_x\left(a(v)v'\right).
\end{equation}
It follows from the De-Giorgi-Nash estimate \cite[Theorem 6.28]{li:parabolic} that $v\in C^{0,\beta}_{\tx{loc}}((0,1)\times(0,T))$ for some $\beta$ depending on $\lambda$ and $\Lambda$, with estimates of the form
$$
\tx{osc}_{Q_r(x,t)}v \leq C\left(\frac{r}{\rho}\right)^\beta\tx{osc}_{Q_{\rho}(x,t)}v
$$
for $r<\rho$ with $Q_\rho(x,t)\subseteq Q_R(0,0)$.  Equation \eqref{eq:gradesthalfR} bounds $\tx{osc}_{Q_{R/2}(0,0)}v$, yielding
\begin{equation}\label{eq:Holderv}
|v(x,t)-v(y,s)|\leq \frac{CM}{R^{1+\beta}}\left(\frac{1+\Lambda}{\lambda}\right)^{1/2}
\left(|y-x|^2+|t-s|\right)^{\beta/2}
\end{equation}
whenever $(x,t)$ and $(y,s)$ are in $Q_{R/4}(0,0)$.
Define $A\in C^1_{\tx{loc}}({\mathbb R})$ by 
\begin{equation}\label{eq:defA}
A(z) = \int_0^z \alpha(s)\,ds.
\end{equation}
The ellipticity assumption implies $\lambda |z|\leq |A(z)|\leq \Lambda|z|$ for all $z$.
Equation \eqref{eq:graddepflow} implies
\begin{equation}\label{eq:evolA}
\partial_t(A(u')) = \alpha(u')\partial^2_x\left(A(u')\right).
\end{equation}
The Kruzhkov argument bounds the gradient $w=\partial_xA(u')=\alpha(u')u''=u_t$:
\begin{equation}\label{eq:boundw1}
|w(x,t)|^2\leq C\sup_{Q_r(x,t)}|A(u')|^2\frac{1+\Lambda}{\lambda r^2}.
\end{equation}
In particular, if $(x,t)\in Q_{R/4}(0,0)$ then $Q_{R/4}(x,t)\subseteq Q_{R/2}(0,0)$ and Equation \eqref{eq:gradesthalfR} yields
\begin{equation}\label{eq:boundw2}
|w(x,t)|^2\leq \frac{C(1+\Lambda)^2\Lambda M^2}{\lambda^2 R^4}.
\end{equation}

Differentiating Equation \eqref{eq:evolA} again gives a divergence-form equation for $w$:
\begin{equation}\label{eq:utt}
\partial_tw = \partial_x\left(\alpha(u')\partial_xw\right).
\end{equation}
Again, the De-Giorgi-Nash estimate implies that
$w\in C^{0,\beta}_{\tx{loc}}((0,1)\times(0,T))$, with 
$$
\tx{osc}_{Q_r(x,t)} w \leq C\left(\frac{r}{\rho}\right)^\beta\tx{osc}_{Q_\rho(x,t)}w
$$
for $r\leq\rho$.  In particular for $(x,t)$ and $(y,s)$ in $Q_{R/8}(0,0)$ we deduce
\begin{equation}\label{eq:Holderw}
\left|w(x,t)-w(y,s)\right|
\leq C\frac{(1+\Lambda)\sqrt{\Lambda} M}{\lambda R^{2+\beta}}
\left(|y-x|^2 + |t-s|\right)^{\beta/2}.
\end{equation}
This gives an estimate on the continuity of $u''$:
\begin{align}
\left| u''(x,t)-u''(y,s) \right|
&\leq \frac{1}{\alpha(u'(x,t))}\left|w(x,t)-w(y,s)\right|
+\left|\frac{\alpha(u'(y,s))}{\alpha(u'(x,t))}-1\right|\left|u''(y,s)\right|\notag\\
&\leq \frac{C(1+\Lambda)\sqrt{\Lambda}M}{\lambda R^2}
\left(\left(\frac{d}{R}\right)^\beta+\frac{1}{\lambda^2}\psi\left(\frac{C\sqrt{1+\Lambda}M}{\sqrt{\lambda}R}\left(\frac{d}{R}\right)^\beta\right)\right),\label{eq:ctsd2u}
\end{align}
where $d=\sqrt{|y-x|^2+|t-s|}$.

\def\cprime{$'$}
\providecommand{\bysame}{\leavevmode\hbox to3em{\hrulefill}\thinspace}
\providecommand{\MR}{\relax\ifhmode\unskip\space\fi MR }
\providecommand{\MRhref}[2]{%
  \href{http://www.ams.org/mathscinet-getitem?mr=#1}{#2}
}
\providecommand{\href}[2]{#2}

\end{document}